\documentclass{LMCS}

\def\doi{8(4:15)2012}
\lmcsheading%
{\doi}
{1--11}
{}
{}
{May\phantom.~20, 2011}
{Nov.~20, 2012}
{}
 
\usepackage{enumerate}
\usepackage{hyperref}

\usepackage{url}
\usepackage{amssymb,latexsym}

\usepackage{graphicx}

\newtheorem{claim}{Claim}
\newtheorem{definition}{Definition}
\newtheorem{theorem}{Theorem}
\newtheorem{lemma}{Lemma}
\newtheorem{corollary}{Corollary}
\newcommand{\IC}{\mathbb{C}}
\newcommand{\ID}{\mathbb{D}}

\newcommand{\IN}{\mathbb{N}}

\newcommand{\IR}{\mathbb{R}}
\newcommand{\IZ}{\mathbb{Z}}
\newcommand{\Dyn}{\mathbb{Y}}

\newcommand{\IY}{\Dyn}

\newcommand{\intervall}{\mathbb{I}}

\newcommand{\normed}{{\mathcal{N}}}

\newcommand{\dom}{\mbox{dom}}

\begin{document}

\title[On Computable Approximations of Landau's Constant]{On Computable Approximations of Landau's Constant}

\author[R.~Rettinger]{Robert Rettinger}	%required
\address{FernUniversit\"at Hagen\\
LG Komplexit\"at und Algorithmen\\
Universit\"atsstrasse 1\\ D-58095 Hagen }	%required
\email{robert.rettinger@fernuni-hagen.de}  %optional
%\thanks{thanks 1, optional.}	%optional

%% etc.

%% required for running head on odd and even pages, use suitable
%% abbreviations in case of long titles and many authors:

%% mandatory lists of keywords and classifications:

\begin{abstract}
We present an algorithm which computes the Landau constant up to any given precision.
\end{abstract}

\maketitle

\section{Introduction}
For most of the mathematical constants used today, like $\pi$, $e$ or $\zeta$, very efficient 
algorithms to approximate these constants are known. There are, however, rare cases such as Bloch's, Landau's or the Hayman-Wu constant, 
where conjectures and rough bounds on such constants are known but the exact value is unknown. Recently, an algorithm to approximate
Bloch's constant was given in \cite{rett08}. We continue this line of research and present a similar algorithm to approximate Landau's constant up
to any precision. 

Because of the close connection between the Landau and Bloch constants we can reuse many of the ideas (and motivations) 
on the computation of the Bloch constant in this paper. A core component of the algorithm, namely the computation of maximum discs in the
image of a (normalised) holomorphic function, is totally different. So injectivity does not play any role any longer. What is more, however,
homotopy methods used in the algorithm for Bloch's constant do not work any longer. This is due to the fact that for discs which need not be schlicht,
several overlapping parts of the pre-image come into play. To decide whether the images do indeed overlap is in general not even computable. 
Though the ideas of the algorithm are inspired by Type-2 theory of effectivity, we formulate our results totally independently of this theory.
In this way we hope that our result is accessible to a wider audience. Finally, we improve even on the
parts which could be taken literally from the algorithm for Bloch's constant to narrow the gap between the theory and possible implementations.

Landau's constant (see \cite{lan29}) gives a quantitative version of the fact that non-constant holomorphic functions are open.
More precisely, it states that for any $r>0$ and any holomorphic function $f$ defined on
a disc $\ID_r(z_0):=\{ z\in\IC\mid |z-z_0|<r\}$ with $f'(z_0)\not=0$ there exists a disc of
radius $|f'(z_0)|\cdot r\cdot c$ inside the image $f(\ID_r(z_0))$, where the constant $c>0$ does not depend on $f$! 
Obviously $c$ is bounded from above; thus the supremum, the so called Landau constant $\lambda$, exists. The best upper 
bound known for $\lambda$,
\[ \lambda \leq \frac{\Gamma(\frac{1}{3})\cdot\Gamma(\frac{5}{6})}{\Gamma(\frac{1}{6})},
\]
is at the same time conjectured to be the exact value of
$\lambda$. However the best lower bound known so far is
\[ \frac{1}{2}< \lambda .
\]
(see \cite{rob38}\ and \cite{rad43}).

Putting this in decimal representation gives
\[0.5<\lambda\leq 0.54325...\]
i.e. all we know is the constant up to $4\cdot 10^{-2}$.

In this paper we will give an algorithm to compute Landau's constant $\lambda$ up to any precision in
the sense that on input $n\in\IN$ some rational number $q$ with
$|q-\lambda|<2^{-n}$ can be computed.

The main idea of our algorithm is to compute for several normalised functions the corresponding
$\lambda$-values. Following the definition, it seems that we have to take the infimum for all normalised
functions, which could not be done in finitely many steps. We will overcome this problem by a compactness argument in Section 5. 
In the next section we recall a few notations and fix the machine model which we will use throughout this paper. 
In Section 3  we introduce a subset of all holomorphic functions such that $\lambda$ is already determined by this smaller, compact class.
In Section 4 we will then show how $\lambda_f$ of a single function can be approximated and, finally, in Section 5 we prove the computability
of Landau's constant.

\section{Preliminaries}\label{sec_prel}

We begin this section with a few remarks on the model underlying our algorithms. Our main algorithm (Section 5) will only compute on finite words, where we assume some straightforward, standard encodings (representations) of the dyadic numbers, i.e. numbers of the form $m\cdot 2^n$ with $m,n\in\IZ$, by such words. We do not need any
operations on infinite words or other structures like $\IR$ or $\IC$, and therefore our algorithms could be implemented by classical Turing machines.
If some value is given as input to our machine, we also say that the machine {\it computes}\ something {\it on}\ this value. 
In addition, to simplify things, we use a second model for intermediate results where we allow infinite sequences as input. 
One can think of this model as a classical Turing machine which can in addition ask some kind of oracle for single elements of the input. The output, however, will always be a finite word which will be returned after a finite number of steps. 
This second model is only used for simplification reasons, and our final algorithm to compute Landau's constant
will not depend on this model. 

Let $\ID_r(z_0):=\{ z\in\IC\mid |z-z_0|<r\}$ denote the open disc of radius $r$ with
centre $z_0$ and let $\overline{\ID}_r(z_0)$ denote its topological closure, i.e. $\overline{\ID}_r(z_0):=\{ z\in\IC\mid |z-z_0|\leq r\}$. To simplify notation we use $\ID_r:=\ID_r(0)$ and $\ID=\ID_1$. A
{\it normalised (holomorphic) function}\ on a domain $D$, $0\in D$, is a holomorphic function $f$ with $f'(0)=1$. The space of normalised functions on $D$ is denoted by $\normed(D)$.
For given domain $A\subseteq \IC$ we denote the radius of the largest disc in $A$ by $l(A)$, i.e.
\[l(A)=\sup \{ r\mid \exists z. \ID_r(z)\subseteq A\}.\] 
Given a holomorphic function $f$ with $D\subseteq\dom(f)$ let $\lambda_f(D)$ denote the radius of the largest disc in $f(D)$, i.e. \[\lambda_f(D)= l(f(D)).\] Finally let $\lambda_f:=\lambda_f(\ID)$ and the {\it Landau constant}\ {$\lambda$} be the infimum of all $\lambda_f$ with $f\in\normed(D)$. 
Obviously, $\lambda$ is already the infimum of all $\lambda_f$ for normalized $f$ with the additional condition $f(0)=0$.

Throughout this paper we use most of the time only very basic results of complex analysis covered by most textbooks (see e.g. \cite{ahl66}).
The only somewhat more advanced result from complex analysis is Koebe's $1/4$ Lemma (see e.g. \cite{rud87}):

\begin{theorem}[Koebe's $1/4$ Lemma]\label{theo_koebe}
Let $f$ be an injective holomorphic function on some disc $\ID_r$. Then $\ID_{r\cdot |f'(0)|\cdot 1/4}(f(0))\subseteq f(\ID_r)$.  
\end{theorem} 

The representations of the objects we use throughout this paper are introduced next, where
we implicitly use some kind of efficiently computable pairing function without further mention.

Let $\Dyn$ denote the class of dyadic numbers, i.e. \[\Dyn := \{ m\cdot 2^{n}\mid n,m\in\IZ\}.\] We could use rational numbers as well but having implementations in
mind, we stick to the efficiently implementable dyadic numbers.
Identifying complex numbers $z=x+\iota\cdot y$ and pairs $(x,y)$, where we denote the imaginary unit by $\iota$,  we can represent
elements of the set $\Dyn[\iota]$ of complex dyadics simply by pairs of dyadic numbers.
In this way, it should be clear that there are efficient algorithms to approximate operations like $\cdot,/,+,-$ on the complex dyadics,
in the sense that, given complex dyadics $z$, $z'\not=0$ and some $n\in\IN$, we can easily compute a dyadic number $y$ such that
$|y-z/z'|<2^{-n}$ etc. 
To simplify notation, we say that we can {\it compute a value $z$ up to precision $2^{-n}$}\ if we can, on input $n$ and $z$, compute a (complex) dyadic $y$ such that $|y-z|<2^{-n}$. 
If we can compute $z$ up to precision $2^{-n}$ for all $n$, then we say that we can compute $z$ up to any precision. 

Let furthermore $\intervall$ denote the class of intervals $[c,c']\times [d,d']$ where $c,c',d,d'$ are dyadics. We consider these intervals as sets of complex numbers.  
Let $\Pi$ be a finite alphabet. Then $\Pi^\ast$ denotes the set of finite sequences (words) over $\Pi$, $\varepsilon$ the empty sequence, and $\Pi^\infty$ the set of infinite sequences over $\Pi$. As usually, we think of $\Pi^\infty$ as a topological space where $\{ w\Pi^\infty\mid w\in\Pi^\ast\}$
is a basis of the topology. It is well known that $\Pi^\infty$ with the above topology is a compact space.

In our algorithm we need to approximate holomorphic functions in a certain function space. 
We will do this by representing this function class by $\Pi^\infty$ for $\Pi = \{ 1,2,3,4\}$.
To 
simplify things we will first define, for given interval $I\in\intervall$, the representation $\Psi_I^\infty$ of the complex numbers in
$I$. To this end, let for a given interval $I=[c,c']\times [d,d']$, the intervals
\[[c,c'']\times [d,d''],\quad [c'',c']\times [d,d''],\quad [c'',c']\times [d'',d'],\quad [c,c'']\times [d'',d'],\] where $c''=(c+c')/2$ and $d''=(d+d')/2$,
be denoted by $I_1$, $I_2$, $I_3$ and $I_4$, respectively.  
%\begin{figure}[h]
%\unitlength1cm
%\includegraphics[scale=0.2]{I.pdf}
%\caption{The definition of $I_1$,...,$I_4$}
%\end{figure}
With these preliminaries, we use sequences $\alpha_0$, $\alpha_1$, $\ldots$ of numbers in $\Pi:=\{1,2,3,4\}$ to represent complex numbers where we
define $\Psi_I:\Pi^\ast\rightarrow \intervall$ by
\[ \begin{array}{lll}\Psi_I(\varepsilon) = I,\\
    \Psi_I(w\alpha) = (\Psi_I(w))_\alpha  
\end{array}\]
for all $I\in\intervall$, $w\in\Pi^\ast$ and $\alpha\in\Pi$, and finally \[\Psi_I^\infty(\alpha_0\alpha_1\ldots)=\bigcap_i \Psi_I(\alpha_0\alpha_1\ldots\alpha_i),\] i.e.
we define complex numbers by suitable nested intervals.

Let now $m_1,m_2,\ldots$ denote a sequence of positive dyadics such that $\sum_{i\geq 1} m_i\cdot\varepsilon^i$ converges for all $\varepsilon\in [0,1)$.
Then the set of sequences $a_1,a_2,\ldots$ of complex numbers where the absolute values of the real and imaginary parts of $a_i$  are bounded by $m_i$ for all $i\geq 1$ defines a class $F_{m_1,m_2,\ldots}$ of holomorphic functions
on $\ID$ by identifying a sequence $a_1$, $a_2$, $\ldots$ with the power series $1+\sum_{i\geq 1} a_i\cdot z^i$. 
We are interested only in power series where the first coefficient is 1 because this class represents exactly the derivatives of normalized holomorphic functions. We use $\Psi$ to define a representation of
these holomorphic functions: Let $t_1$, $t_2$, $\ldots$ be a sequence of natural numbers such that any natural number is encountered infinitely often,
e.g. $0$, $0$, $1$, $0$, $1$, $2$, $0$, $1$, $2$, $3$, $0$, $\ldots$. Then for a given sequence $\alpha_1$, $\alpha_2$, $\ldots$ with $\alpha_i\in\Pi$ and
any $n\in\IN$ the sequence  $t_1$, $t_2$, $\ldots$ defines the subsequence $\alpha_{n_1}$,  $\alpha_{n_2}$, $\ldots$ of all those elements $\alpha_{n_j}$
such that $t_{n_j}=n$. In this way we can define a mapping \[\Psi_{m_1,m_2,\ldots}^\infty:\Pi^\infty\rightarrow F_{m_1,m_2,\ldots}\] by 
\[ \Psi_{m_1,m_2,\ldots}^\infty(\alpha_1\alpha_2\ldots)=1+\sum_{n\geq 1} \Psi_{[-m_n,m_n]\times [-m_n,m_n]}^\infty(\alpha_{n_1}\alpha_{n_2}\ldots)\cdot z^n.\]

To simplify notations we denote, for given holomorphic function $f:\ID\rightarrow \IC$, the anti-derivative of $f$, which maps $0$ to $0$, by $\int f$, 
i.e. $\int f:\ID\rightarrow\IC$, $(\int f)(0)=0$ and $(\int f)' = f$. 
Notice that $\int \Psi_{m_1,m_2,\ldots}^\infty(\alpha_1\alpha_2\ldots)\in\normed(\ID)$ with the above settings.
A proof of the following theorem can for example be found in \cite{muel87}\ or \cite{rett08b}.

\begin{theorem}\label{theo_L1}
Let $\alpha_1$, $\alpha_2$, $\ldots$, $m_1$, $m_2$, $\ldots$ and $t_1$, $t_2$, $\ldots$ be as above. Furthermore let $b_1$, $b_2$, $\ldots$ be
a sequence of positive dyadic numbers such that with \[f=\Psi_{m_1,m_2,\ldots}^\infty(\alpha_1\alpha_2\ldots),\] for all $n\geq 1$ we have
\[ \forall z\in\ID_{1-2^{-n}}.|f(z)|<b_n.\]
Then, on input $r\in\IY\cap(0,1)$, $z\in\IY[\iota]$ and \[\alpha_1, \alpha_2, \ldots,\quad m_1, m_2, \ldots,\quad t_1, t_2, \ldots,\quad b_1, b_2, \ldots\mbox{ and }n\in\IN,\] we can compute $\int f(z)$, $f(z)$  and $f^{(n)}(z)$
on $\ID$ up to any precision.
%(see e.g. [Rettinger continuation???][Mueller???]).
\end{theorem}

\section{$\lambda$-bounding Functions}\label{sec_3_2}

Similar to the $\beta$-bounding functions in \cite{rett08}\ we define here a class $F_\lambda$ of holomorphic functions which
determines $\lambda$ in the sense that $\lambda=\inf_{f\in F_\lambda}\lambda_f$. Unlike \cite{rett08}\ we use different bounds on
the coefficients of the corresponding power series and, furthermore, use the compactness of $\Pi^\infty$ together with the
mapping $\Psi^\infty$ rather than compactness of the class $F_\lambda$ itself.

To start with we will prove suitable upper bounds on (the derivative of) functions $f$ whose $\lambda$-values approximate $\lambda$: 

\begin{lemma}\label{lemma_3_2}
Let $c>1$ be given. Then for any $f\in\normed(\ID)$ and $w\in\ID$ with
$|f'(w)|\geq c /(1-|w|^2)$ we have
\begin{equation*} \lambda_f\geq c\cdot \lambda.\end{equation*}
\end{lemma}

\begin{proof} 
Let $w\in\ID$ and $f\in\normed(\ID)$ with $|f'(w)|\geq c /(1-|w|^2)$ be given. 
Let $g:\ID\rightarrow\ID$ be the automorphism \[z\mapsto (z+w)/(1+\overline w z).\]
Then $f\circ g(0) = f(w)$, $(f\circ g)'(0) = f'(w)\cdot (1-|w|^2)$, and 
\[h:=\frac{1}{f'(w)\cdot (1-|w|^2)}  f\circ g\] is a normalised
function. Therefore 
\[\frac{1}{c}\cdot \lambda_{f\circ g}\geq \frac{1}{|f'(w)|\cdot (1-|w|^2)}\cdot \lambda_{f\circ g}\geq \lambda_h\geq \lambda.\] 
As $f(\ID)=f\circ g(\ID)$ the statement of the lemma follows. 
\end{proof}

To simplify things we fix some value $c>1$, say $c=1+2^{-100}$ for the time being. To find $\lambda$ it suffices to consider all $\lambda_f$
of all normalised functions $f$ with $|f'(z)|\leq c/(1-|z|^2)$ for all $z\in\ID$. This immediately gives us a bound on the coefficients of
the corresponding power series:

\begin{lemma}\label{lemma_3_3}
Let $f$ be a normalised function such that 
$|f'(z)|\leq c/(1-|z|^2)$ for all $z\in\ID$. Then $f'(z)=\sum_n a_n\cdot z^n$ with
\begin{enumerate}[(a)]
\item $a_0=1$ and
\item $|a_n|\leq c\cdot e\cdot (n+2)/2$ for all $n\geq 1$.
\end{enumerate}
(Here $e$ denotes the Euler constant.)
\end{lemma}
\proof
Item (a) is obvious because $f$ is assumed to be normalised.

By the Cauchy inequality we have for given $n$ and $r\in (0;1)$
\[ |a_n| \leq \frac{1}{r^n} \sup_{|z|=r} |f'(z)|\leq \frac{1}{r^n}\cdot \frac{c}{1-r^2}. \]
Choosing $r:=\sqrt{n/(n+2)}$ we get
\[ |a_n| \leq \frac{1}{r^n}\cdot \frac{c}{1-r^2}\leq c\cdot\left(1+\frac{2}{n}\right)^{\frac{n}{2}}\cdot \frac{n+2}{2}\leq c\cdot e\cdot \frac{n+2}{2}.\eqno{\qEd} \]

Let, for $n\geq 1$, $m_n$ be some dyadic approximation such that \[c\cdot e\cdot (n+2)/2\leq m_n \leq c\cdot e\cdot (n+2)/2 + 2^{-n}.\]
Then we have for given $\varepsilon\in (0;1)$
\[ \begin{array}{lll}1+\sum_{n\geq 1} m_n\cdot \varepsilon^n & \leq & \frac{c\cdot e}{2} \cdot (\sum_n \varepsilon^n + \sum_n (n+1)\cdot\varepsilon^n) + \sum_n 2^{-n} -c\cdot e\\[0.2cm]
& \leq & \frac{c\cdot e}{2} \cdot 
 \left(\frac{1}{1-\varepsilon}+\frac{1}{(1-\varepsilon)^2}\right) + (2-c\cdot e).\end{array}\]
A similar bound does also hold for the corresponding derivatives:
\[ \sum_{n\geq 1} n\cdot m_n\cdot \varepsilon^{n-1}\leq \frac{c\cdot e}{(1-\varepsilon)^3} + \frac{c\cdot e}{2\cdot (1-\varepsilon)^2} + 2. 
\]
Thus we can easily compute a sequence $b_1$, $b_2$, $\ldots$
of bounds such that Theorem \ref{theo_L1}\ can be applied. To simplify notations let \[\Psi^\infty:=\Psi_{m_1,m_2,\ldots}^\infty.\] 
In the definition of
$\Psi^\infty$ above the coefficients of the corresponding power series in $\Psi^\infty(\Pi^\infty)$ can get larger than we considered so far
because, to simplify things, we bound the real and imaginary parts of the $n$-th coefficient by $m_n$, not the absolute value of the coefficient itself. Thus an additional factor $\sqrt{2}$ is sufficient in the following corollary. 
Due to the structure of our algorithm, however, we can get rid of this additional factor in implementations by 
restricting $\Pi^\infty$ to a proper subclass by a simple test for the absolute value of the coefficients. 

We can summarize the essence of this
section by the following corollary:

\begin{corollary}\label{cor_3_2}
Let $c>1$ and let, for $n\geq 1$, $m_n$ be some dyadic number such that
\[c\cdot e\cdot \frac{n+2}{2}\leq m_n \leq c\cdot e\cdot \frac{n+2}{2} + 2^{-n}.\] 
 Then 
\begin{enumerate}[(a)]
\item $\lambda = \inf \{ \lambda_{\int f}\mid f\in \Psi^\infty(\Pi^\infty)\}$,
\item we can compute on given dyadic complex number $z\in\ID$, $\alpha\in\Pi^\infty$, $n\in\IN$ and $r\in\IY\cap(0;1)$ 
\begin{enumerate}[(i)]
\item $f^{(n)}(z)$ and 
\item upper bounds \[\mu_r' := \sqrt{2}\cdot\left(c\cdot \frac{e}{2} \cdot \left(\frac{1}{1-r}+\frac{1}{(1-r)^2}\right) +  (2-c\cdot e)\right)\] 
            and
                         \[\mu_r'' := \sqrt{2}\cdot\left(\frac{c\cdot e}{(1-r)^3} + \frac{c\cdot e}{2\cdot (1-r)^2} + 2\right)\]
            of $\sup \{ |f'(z)|\mid z\in \ID_r\}$ and $\sup \{ |f''(z)|\mid z\in \ID_r\}$, respectively,\\ 
\end{enumerate}
up to any precision, where $f:=\int \Psi^{\infty}(\alpha)$.
\end{enumerate}  
\end{corollary}

\section{The $\lambda$-value of a single Function}\label{sec_3_3}

In this section we will show how $\lambda_f$ can be approximated. More precisely, we will give lower bounds on $\lambda_f$ 
which are at the same time approximately upper bounds for $\lambda$. This will be enough to compute $\lambda$ up to any precision in the end. 
To this end, let $\alpha_1\alpha_2\ldots$ be a fixed sequence in $\Pi^\infty$ which will be fed to all algorithms considered in this section as input. Furthermore, we 
will denote the corresponding functions $\Psi^\infty(\alpha_1\ldots)$ and $\int\Psi^\infty(\alpha_1\ldots)$ by $f'$ and $f$, respectively.

As in the case of Bloch's constant, it suffices to search large discs in the image of a proper sub-domain of the unit disc to approximate $\lambda_f$.
One advantage of this is that we have to search only on a bounded image for such discs. The following lemma can be proven by a simple transformation
of holomorphic functions $f(z)\mapsto \frac{1}{r} f(r\cdot z)$ (see e.g. \cite{con78}).

\begin{lemma}\label{lemma_3_1}
Let $r$ with $0<r<1$ and a domain $D$ with $\ID_r\subseteq
D\subseteq\ID$ be given. Then
\[ r \lambda\leq \lambda_f(D)\leq\lambda_f.
\]
\end{lemma}

Now, approximating the image $f(\ID_r)$ can be done straightforwardly. 
Furthermore, with quite basic methods, the largest disc inside such an approximation could be found easily.
We are doing exactly this by the $\varepsilon$-covering grids which we introduce next.
However, an approximation will not necessarily mean that we have an approximation of $\lambda_f$:
there indeed exist examples, where small changes of $f$ can lead to a large change in $\lambda_f$.
% (see Figure ???).
That means that the homotopic methods used in \cite{rett08}\ cannot be used here. Instead, we will show that the approximations 
we get by $\varepsilon$-covering grids for small $\varepsilon$ are still suitable to bound $\lambda$ from above.

\begin{definition}\label{def_3_1}
Let $\varepsilon>0$ and a bounded subset $A\subseteq\IC$ be given. Then an
{\it $\varepsilon$-covering grid}\ of $A$ is a tuple $(\varepsilon,\delta, G)$ where $0<\delta\leq\varepsilon/4$ is some dyadic number and
$G$ is a non-empty, finite subset of $\delta\IZ\times\delta\IZ$ such that
\begin{enumerate}[(a)]
\item $A\cap \left(\delta\IZ\times\delta\IZ\right) \subseteq G$ and
\item $\forall z\in G.\exists {a\in A}.|z-a|\leq\varepsilon/4$.\\
\end{enumerate}
%
%The set of $\varepsilon$-covering grids of $A$ is denoted by $\Cov_\varepsilon(A)$.
\end{definition}

Notice, that for any $\varepsilon > \varepsilon' >0$ and any $\varepsilon'$-covering grid $(\varepsilon',\delta, G)$ of a set $A$, the
tuple $(\varepsilon, \delta, G)$ is an $\varepsilon$-covering grid of $A$. In this sense any $\varepsilon'$-covering grid of a set $A$ is
also an $\varepsilon$-covering grid.

Following the notation of Section \ref{sec_prel}, $l(A)$ denotes the radius of the largest disc inside a domain $A\subseteq\IC$. Furthermore let \[l(\varepsilon,\delta, G):=\delta+\max_{z\in G}\min_{y\in (\delta\IZ\times\delta\IZ)\setminus G} |z-y|\] for
an $\varepsilon$-covering grid $(\varepsilon,\delta, G)$. Then the easy to compute value $l(\varepsilon,\delta, G)$ gives us an approximation
of the largest disc inside the "covered set" \[D(\varepsilon,\delta, G) := \bigcup_{z\in G} \ID_{\frac{3}{4}\varepsilon}(z).\] More precisely we get the following result:

\begin{lemma}\label{lemma_3_2}
Let $(\varepsilon,\delta, G)$ be an $\varepsilon$-covering grid of $A\subseteq \IC$. Then
\[ l(A)\leq l(\varepsilon,\delta, G)\leq l\left(D(\varepsilon,\delta,G)\right).\]
\end{lemma}

\begin{proof}
We start by proving the left inequality: Let $\ID_r(z)$ be some disc with $\ID_r(z)\subseteq A$.
We can assume that $r>\delta$ because otherwise we have $r\leq\delta\leq l(\varepsilon,\delta, G)$ anyway. 
There exists some $y\in G$ such that $|y-z|\leq \delta$.
Let furthermore $x\in \delta\IZ\times\delta\IZ$ be some point with $x\not\in G$. Then we have $\overline{\ID}_{|x-y|}(y)\not\subseteq A$ and
thus $\overline{\ID}_{|x-y|+\delta}(z)\not\subseteq A$. Therefore  $|x-y|+\delta\geq r$ and the statement follows.

The right inequality follows immediately from the following claim:

\begin{claim} \label{claim_n_1}
Let $x\in G$ be given and $r:= \delta+\min_{z\in (\delta\IZ\times\delta\IZ)\setminus G} |x-z|$. Then
$ \ID_r(x)\subseteq D(\varepsilon,\delta,G)$. 
\end{claim}

Let us assume that there exist $y\not\in D(\varepsilon,\delta,G) = \bigcup_{z\in G} \ID_{\frac{3}{4}\varepsilon}(z)$
such that $|x-y|< r$.
Then there also exists some point $z'\in\delta\IZ\times\delta\IZ$ such that $|x-z'|<r$ and $|z'-y|< \sqrt{2} \delta$. As $\sqrt{2} \delta <3/4\cdot \varepsilon$, $z'$ cannot belong to $G$ because otherwise $y$ would belong to $\bigcup_{z\in G} \ID_{\frac{3}{4}\varepsilon}(z)$. 
\begin{figure}[h]
\unitlength1cm
\includegraphics[scale=0.3]{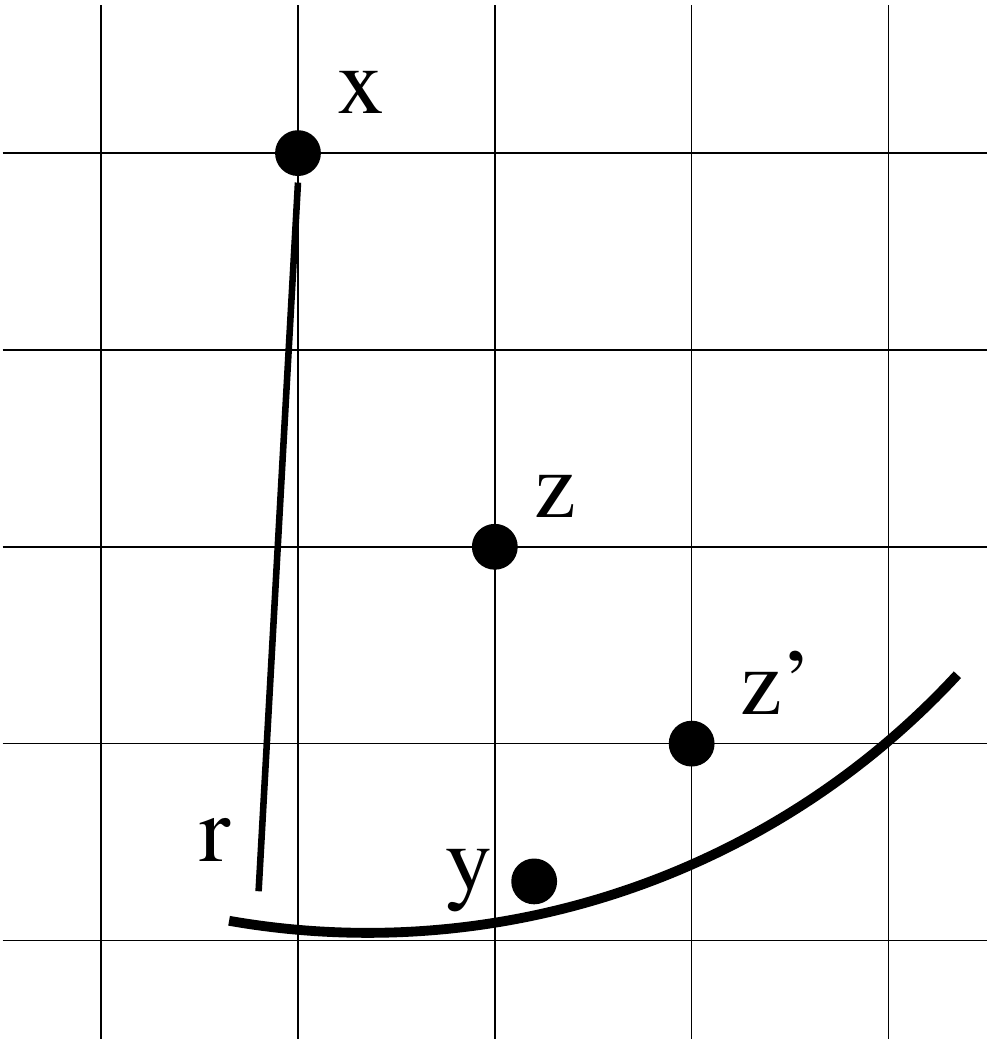}
\caption{Proof of Claim \ref{claim_n_1}}
\end{figure}
Furthermore, as $z'\not= x$ and $x,z'\in\delta\IZ\times\delta\IZ$ there 
exists some $z\in \delta\IZ\times\delta\IZ$ such that $|z-z'|\leq\sqrt{2}\cdot\delta$ and $|z-x|\leq |z'-x|-\delta < r-\delta$, which means
$z\in G$ by the definition of $r$ and furthermore \[|z-y|\leq 2\cdot\sqrt{2}\cdot\delta \leq 2\cdot\frac{\sqrt{2}}{4} \cdot \varepsilon < \frac{3}{4}\cdot\varepsilon\] in contradiction
to $y\not\in \bigcup_{z\in G} \ID_{\frac{3}{4}\varepsilon}(z)$.
\end{proof}

Besides, an $\varepsilon$-covering grid of $f(\ID_r)$ can be easily computed:

\begin{lemma}\label{lemma_3_3}
For given dyadic numbers $\varepsilon>0$ and $0<r<1$ as input, we can compute an $\varepsilon$-covering grid of $f(\ID_r)$. 
\end{lemma}

\proof
Following the notation of Theorem 2 we can compute a dyadic upper bound $\mu_{r}'$ on the maximum 
of the values $|f'(z)|$ with $z\in\ID_{r}$. Furthermore let $\delta_D>0$ be a dyadic 
lower bound on $\varepsilon/(16\cdot\mu_{r}')$ and \[G_D:=\ID_r\cap (\delta_D\IZ\times\delta_D\IZ).\] Notice that $0\in G_D$ and thus $G_D\not=\emptyset$. Furthermore we can compute for every $z\in G_D$ some approximation $d_z$ with $|d_z-f(z)|\leq\varepsilon/16$. 

Then $(\varepsilon, \varepsilon/4, G)$ is an $\varepsilon$-covering grid of $f(\ID_r)$, where
\[ G = \bigcup_{z\in G_D} \{ y\in \frac{\varepsilon}{4}\IZ\times\frac{\varepsilon}{4}\IZ\mid |d_z-y|\leq \frac{3}{2} \cdot \frac{\varepsilon}{8}\}. \]

To see item (a) of Definition \ref{def_3_1}\ let $y\in f(\ID_r)\cap \left(\frac{\varepsilon}{4}\IZ\times\frac{\varepsilon}{4}\IZ\right)$ be given. Furthermore let $x\in\ID_r$ such that $f(x)=y$. Then there exists some $z\in G_D$ with $|z-x|<\sqrt{2}\cdot\delta_D$ and we get 
\[ \begin{array}{lll}|y-d_z|& = & |f(x)-d_z|\\
& \leq & |f(x)-f(z)|+|f(z)-d_z|\\
& \leq & \mu_{r}'\cdot |x-z| + \frac{\varepsilon}{16}\\ 
& \leq & \sqrt{2}\cdot\frac{\varepsilon}{16} + \frac{\varepsilon}{16}\\
& \leq & \frac{3}{2}\cdot\frac{\varepsilon}{8} \end{array}\]

To see item (b) of Definition \ref{def_3_1}\ let now $y\in G$ be given. Then there exist $z\in G_D$ such that $|d_z-y|\leq (3/2)\cdot (\varepsilon/8)$ and
 we have \[ \inf_{a\in f(\ID_r)} | y - a| \leq |y-f(z)|\leq |d_z-y| + | f(z)-d_z| \leq \frac{3}{2}\cdot\frac{\varepsilon}{8} + \frac{\varepsilon}{16} = \frac{\varepsilon}{4}.\eqno{\qEd}\]

Up to this point we have followed the naive way of computing large discs in approximations of $f(\ID_r)$. We have already argued that this
will not necessarily give us appropriate bounds on $\lambda_f$. The main step we will take next is to show that for suitable $\hat r > r$
and suitably small $\varepsilon$ we get that for any $\varepsilon$-covering grid $(\varepsilon,\delta,G)$ of $f(\ID_r)$ we have $D(\varepsilon,\delta,G)\subseteq f(\ID_{\hat r})$, which in the end allows us to show that the bound given by $l(\varepsilon,\delta,G)$ is actually not too bad.

\begin{lemma}\label{lemma_3_4}
On given dyadic number $0<r<1$ we can compute dyadic numbers $\varepsilon>0$ and $\hat r$ with $r<\hat r<1$ such that for any $\varepsilon$-covering grid $(\varepsilon,\delta,G)$ of $f(\ID_r)$ we have
\[ \bigcup_{z\in G} \ID_{\frac{3}{4}\varepsilon}(z) \subseteq f(\ID_{\hat r}).\]
\end{lemma}

\begin{proof}
As $f$ is not constant, we can compute $\rho>0$ and $\hat{r}$ such that $0<r<\hat{r}<1$ and $|f'(z)|\geq\rho$ for all $z\in\IC$ with $|z|=\hat r$. Furthermore, following the notation of Corollary 1, we can compute a dyadic upper bound  $\mu_{\hat r}''$ on the maximum of $|f''(z)|$ for $z\in \ID_{\hat r}$. Let $\Delta$ be a dyadic number with \[0 < \Delta \leq \rho/(4\mu_{\hat r}'')\mbox{ and }2\cdot \Delta < \hat r - r.\] 
Then we have $|f'(z)|\geq \rho/2$ for all $z\in \ID_{\hat{r}}\setminus\ID_{\hat r -2\Delta}$ and, with $\overline r := \hat{r} - \Delta$ we have $0<r<\overline r< \hat r <1$. By choice of $\Delta$ we do have actually more: $f$ is injective on $\ID_{\Delta/2}(z)$ for each 
$z\in \IC$ with \[\overline r -\Delta/2\leq |z|\leq \overline r+\Delta/2.\] Thus we can apply Koebe's $1/4$-Lemma to get the following claim:

\begin{claim}\label{claim_1}
For each $z\in \IC$ with $\overline r -\Delta/2 \leq |z| \leq \overline r + \Delta/2$ and each $y$ with $|f(z)-y|< \rho\cdot \Delta/16$ we have $y\in f(\ID_{\hat r})$.
\end{claim} 

Assume now, that there exists some $z_0$ with $|z_0|<\overline r - \Delta/2$ and some $y_0$ with $|y_0-f(z_0)|<\rho\cdot \Delta/16$ and 
$y_0\notin f(\ID_{\overline r})$. Then for the holomorphic function $h:\ID\rightarrow\IC$, determined by $h(z):= y_0 - f(z)$ for all $z\in\ID$, we have $h(z)\not= 0$ for all $z\in\ID_{\overline r}$, i.e. $|h|$ takes its minimum on the boundary $\partial \ID_{\overline r}=\{ z\mid |z|=\overline r\}$. Therefore there exists some $z$ with $|z|=\overline r$ and \[|h(z)|\leq |y_0 - f(z_0)|< \rho\cdot \frac{\Delta}{16}.\] Thus by Claim \ref{claim_1}\ we have $y_0\in f(\ID_{\hat r})$,
i.e. we can extend Claim \ref{claim_1}\ as follows:

\begin{claim}\label{claim_2}
For each $z\in \IC$ with $ |z| \leq \overline r + \Delta/2$ and each $y$ with $|f(z)-y|< \rho\cdot \Delta/16$ we have $y\in f(\ID_{\hat r})$.
\end{claim}

Let now $\varepsilon:=\rho\cdot \Delta/16$, $(\varepsilon,\delta,G)$ be some $\varepsilon$-covering grid of $f(\ID_{r})$ and 
$x\in \bigcup_{z\in G} \ID_{\frac{3}{4}\varepsilon}(z)$. Then there exists some $z\in G$ with $|z-x|<(3/4)\cdot \varepsilon$ and,
by condition (b) in Definition 1, there exist $y\in f(\ID_r)$ such that $|z - y|\leq (1/4)\cdot\varepsilon$.
Thus we have $|x-y|< (1/4)\cdot\varepsilon + (3/4)\cdot\varepsilon =
\rho\cdot \Delta/16$ and, by Claim \ref{claim_2}\ we have $x\in f(\ID_{\hat r})$ which proves the lemma.  
\end{proof}

Finally, we can combine the results of this section as follows:

\begin{corollary}
Given $\alpha_1\alpha_2\ldots$ in $\Pi^\infty$ and $n\in\IN$, we can compute a dyadic number $l$ such that $(1-2^{-n})\lambda\leq l\leq \lambda_f$, where $f=\int \Psi^\infty(\alpha_1\alpha_2\ldots)$.
\end{corollary}

\proof
Given some $n\in\IN$ we can compute by Lemma \ref{lemma_3_4}\ first some $\varepsilon>0$ and $\hat r<1$ with $\hat{r}>1-2^{-n}$ 
such that for any $\varepsilon$-covering grid $(\varepsilon,\delta,G)$ of
$f(\ID_{1-2^{-n}})$ we have $D(\varepsilon,\delta,G)\subseteq f(\ID_{\hat r})$. Then, by Lemma \ref{lemma_3_3}\ we can compute an $\varepsilon$-covering grid $(\varepsilon,\delta,G)$ of $f(\ID_{1-2^{-n}})$
and compute $l:=l(\varepsilon,\delta,G)$. Finally, by Lemma \ref{lemma_3_2}\ and Lemma \ref{lemma_3_1}\ we get
\[ (1-2^{-n})\cdot \lambda\leq \lambda_f(\ID_{1-2^{-n}})\leq l\leq l(D(\varepsilon,\delta,G))\leq \lambda_f(\ID_{\hat r})\leq \lambda_f.\eqno{\qEd}
\]

\section{The Main Theorem}\label{sec_3_4}

The proof of our main theorem can now be simply done by covering the space of $\lambda$-bounding
functions by neighbourhoods given by the algorithm of the previous section.

\begin{theorem}\label{theo_3_1}
We can compute approximations to the Landau constant up to any precision.
\end{theorem}

\begin{proof} 
Let $n\in\IN$ be given. We proceed in steps $t=0,1,\ldots$ as follows. For each $t$ we 
consider the set \[W_t:=\{ w1^\infty\mid w\in \Pi^t\},\] i.e. the set of infinite words starting with an
arbitrary (finite) word followed by infinitely many 1. Now we apply the algorithm given in Corollary 2 to each word 
$\omega$ in $W_t$ and receive an value $l(\omega)$ such that \[(1-2^{-n})\lambda\leq l(\omega)\leq \lambda_{\int \Psi^{\infty}(\omega)}.\]
If for any $\omega\in W_t$ information on any symbol outside the leading $t$ symbols of $\omega$ are asked by the algorithm, then
we go to the next step $t+1$. Otherwise $l:=\inf_{\omega\in W_t} l(\omega)$ is an appropriate approximation of $\lambda$, i.e.
$(1-2^{-n})\lambda\leq l\leq\lambda$.

The latter is obvious, because, as none of the algorithms asks any information outside the leading words $w\in\Pi^t$, the
results on $w1^\infty$ equals the results on $w\nu$ for all $\nu\in\Pi^\infty$. Thus $l$ is indeed the infimum over all approximations 
to $\lambda_{\int f}$ of all $f\in\Psi^\infty(\Pi^\infty)$.

It remains, therefore, to prove that for each $n$ there does indeed exist some $t$ such that our algorithm stops after this step.
This can be seen as follows: As the algorithm of Corollary 2 stops on all $\alpha\in\Pi^\infty$, there exist for
each $\alpha$ some $n(\alpha)$ such that no information is asked by the algorithm on any symbol beside the leading $n(\alpha)$ symbols
of $\alpha$. Thus the computation of the algorithm is identical for all words which coincides with $\alpha$ on the first $n(\alpha)$
symbols. This means that for each $\alpha$ there exists an open neighbourhood of $\alpha$ in $\Pi^\infty$ such that the algorithm
is identical on any word in this neighbourhood. As $\Pi^\infty$ is compact, there is a finite sub-covering of $\Pi^\infty$ by such
neighbourhoods. Let $\alpha^1$,...,$\alpha^m$ be the corresponding words. Then the algorithm will stop in step \[t:=\max \{ n(\alpha^i)\mid 1\leq i\leq m\}\] or before.   
\end{proof}

The most interesting open problem is whether the conjectures on the Landau and Bloch constant hold. 
If so, the constant can clearly be computed in polynomial time. To this end, our algorithm and the algorithm given in \cite{rett08}\ 
can present holomorphic functions which are very near the optimum for $\beta$- and $\lambda$-values, thus
giving possibly new insights on the kind of functions involved.

Concerning our algorithm the main intriguing problem is to improve the complexity bound, which is
roughly double exponential, to an acceptable running time. The problem on reducing the time
complexity of our algorithm is that the functions we have to consider can explode when reaching the
boundary $\partial\ID$, where evaluation can be quite expensive.\\[3cm]

%\section*{Acknowledgements}


\begin{thebibliography}{[MT1]}
\bibitem[Ahlfors 1966]{ahl66} Ahlfors, L., {\it Complex analysis}, McGraw-Hill, New York, 1966.
% \bibitem[Ahlfors and Grunsky 1937]{AG37} Ahlfors, L. and Grunsky, H., {\it \"Uber die Bloch'sche Konstante}, Math.Z. 42, pp 671 - 673, 1937.
% \bibitem[Bloch 1929]{blo29} Bloch, A., {\it Les theoremes {M. Valiron} sur les fonctions entieres et la theorie de l'uniformisation}, Annales de la {Faculte} des sciences de {l'Universite} de {Toulouse} 17, pp 1 -22, 1929.
% \bibitem[Chen and Gauthier 1996]{CG96}  Chen, H. and Gauthier, P.M., {\it On Bloch's constant}, J. Anal.Math. 69, pp 275 - 291, 1996.
\bibitem[Conway 1978]{con78} Conway, J.B., {\it Functions in One Complex Variable}, Graduate Texts in Mathematics 11, Springer Verlag, Berlin/Heidelberg, 1978.
\bibitem[Landau 1929]{lan29} Landau, E., {\it {\"U}ber die {B}lochsche {K}onstante und zwei verwandte
{W}eltkonstanten}, Mathematische Zeitschrift, vol. 30, pp 608 - 634, 1929.
\bibitem[Mueller 1987]{muel87} M\"uller, N., {\it Uniform computational complexity of Taylor series}, ICALP 1987, LNCS 267, Springer Verlag, pp 435 - 444, 1987.
\bibitem[Rademacher 1943]{rad43} Rademacher, H., {\it On the Bloch-Landau constant}, American Journal of Mathematics 65, pp 387 - 390, 1943. 
% \bibitem[Remmert 1991]{rem91} Remmert, R., {\it Funktionentheorie {II}}, Springer-Verlag, 1991.
\bibitem[Rettinger 2008]{rett08} Rettinger, R., {\it Bloch's Constant is Computable}, Journal of Universal Computer Science 14(6), pp 896 - 907, 2008.
\bibitem[Rettinger 2008b]{rett08b} Rettinger, R., {\it Lower Bounds on the Continuation of Holomorphic Functions}, CCA 2008, ENTCS 221, pp 207 - 217, 2008.
\bibitem[Robinson 1938]{rob38} Robinson, R.M., {\it unpublished manuscript}, 1938.
\bibitem[Rudin 1987]{rud87} Rudin, W., {\it Real and Complex Analysis}, Series in Higher Mathematics (3 ed.), McGraw-Hill, 1987.
% \bibitem[Weihrauch 2000]{Wei00} Weihrauch, K., {\it Computable Analysis}, Springer-Verlag Berlin/Heidelberg, 2000.
\end{thebibliography}
\end{document}